\documentclass[11pt]{amsart}

\usepackage{verbatim,url}
\usepackage{amsmath}
\usepackage{enumerate,supertabular}
\usepackage{amssymb}
\usepackage{color}
\usepackage[active]{srcltx}
\numberwithin{equation}{section}

\usepackage[all]{xy}
\usepackage[active]{srcltx}

\renewcommand{\!}{\negmedspace\negmedspace}

\theoremstyle{definition}

\newcommand{\Q}{\mathbb{Q}}

\setbox0=\hbox{$+$}
\newdimen\plusheight
\plusheight=\ht 0
\newcommand\+{\;\lower\plusheight\hbox{$+$}\;}

\setbox0=\hbox{$-$}
\newdimen\minusheight
\minusheight=\ht0
\renewcommand\-{\;\lower\minusheight\hbox{$-$}\;}

\setbox0=\hbox{$\cdots$}
\newdimen\cdotsheight
\cdotsheight=\plusheight
\newcommand\cds{\lower\cdotsheight\hbox{$\cdots$}}

\newcommand{\Z}{\mathbb{Z}}

\newtheorem{theorem}{Theorem}   

\newtheorem{lemma}[theorem]{Lemma}
\newtheorem{proposition}[theorem]{Proposition}
\newtheorem{corollary}[theorem]{Corollary}
\newtheorem{definition}[theorem]{Definition}
\newtheorem{remark}[theorem]{Remark}

\begin{document}

\title[Ramanujan Invariants]{Ramanujan Invariants for discriminants congruent to 
$\mathbf{5\;mod \;24}$}  
\author[ Konstantinou,  Kontogeorgis]{ Elisavet Konstantinou and  Aristides Kontogeorgis}

\address{
Department of Information and Communication Systems Engineering, University of the Aegean, 83200
Karlovassi, Samos, Greece} \email{ekonstantinou@aegean.gr}

\address{
Department of Mathematics, University of Athens, Panepistimioupolis Zografou 15784,
 Athens, Greece.
} \email{kontogar@math.uoa.gr}

\begin{abstract}{
In this paper we compute the minimal polynomials of Ramanujan values $27t_n^{-12}$
for discriminants $D \equiv 5 \bmod {24}$. Our method is based on Shimura Reciprocity Law
as which  was made computationally explicit by A.Gee and P. Stevenhagen in \cite{GeeStevenhagen}. However, 
since these Ramanujan values are not class invariants, we present a modification of the 
method used in \cite{GeeStevenhagen} which can be applied on modular functions that do not
necessarily yield class invariants.}
\end{abstract}

\dedicatory{Dedicated to Professor Jannis A. Antoniadis on the occasion
     of his 60th birthday}

\maketitle

\section{Introduction}
It is known that the ring class field of imaginary quadratic orders  can  
be  generated by evaluating the $j$-invariant at certain algebraic integers.
Several other modular functions, like the 
Weber functions \cite{YuiZagier} can also be used for the generation
of the ring  class field. In \cite{GeeBordeaux,GeeStevenhagen} 
A. Gee and P. Stevenhagen developed 
a method based 
on Shimura reciprocity theory, in order to check whether a modular function 
gives rise to a class invariant and in the case  it does, they provided 
a method for the efficient computation of the corresponding minimal polynomial.
This method was generalized further in \cite{steven2} to handle the  ring class 
fields case as well. 
Shimura reciprocity law relates the Galois group of the ray class field $H_{N,\mathcal{O}}$ of conductor 
$N$ over the ring  class field $H_\mathcal{O}$ of the order $\mathcal{O}$,  to the group 
$G_N=(\mathcal{O}/N\mathcal{O})^*/\mathcal{O}^*$. 
In our study we encounter  modular functions of level $72$, and the 
structure and order of the group $G_{72}$ depends on the decomposition of 
$2 \mathcal{O}, 3\mathcal{O}$  as
product of prime ideals in $\mathcal{O}$. 
If the ideals $2,3$ do not remain inert simultaneously then a variety of modular functions 
like the Weber functions, double eta functions  etc. can be used for constructing the ring class field.

Let $K_n=\mathbb{Q}(\sqrt{-n})$ be an imaginary quadratic number field such that  
$n\equiv 19 \bmod 24$ and assume that $\mathcal{O}\subset K$. If $n$ is squarefree then $D=-n$ 
is a fundamental discriminant of $K_n$. 
In this paper we will treat the case when $2,3$ both remain inert,  i.e., 
$2 \mathcal{O}$ and $3\mathcal{O}$ are prime ideals of $\mathcal{O}$.  
In this article we are interested in the $-n \equiv 1\bmod 4$ case so we set
$\theta_n=\frac{1}{2} + i \frac{\sqrt{n}}{2}$ and we consider the order 
$\mathcal{O}=\Z[\theta_n]$ which is a maximal order if $n$ is squarefree. 
Notice that the case $n\equiv 19 \bmod 24$ is the only case where $2,3$ remain inert.

The authors used the method of A. Gee and P. Stevenhagen  \cite{KoKo,KoKo2} in order to 
construct the minimal polynomials of the Ramanujan values $t_n$ for $n\equiv 11 \bmod 24$ 
proposed by S. Ramanujan in his third notebook, pages 392 and 393 
in the pagination of \cite[vol.\ 2]{RamNotebooks}. For a definition of $t_n$, see section \ref{sec3}
eq. (\ref{gg}).
The values $t_n$ were proven to be class invariants for $n\equiv 11 \bmod 24$ by    
B. Berndt and H.H.Chan in \cite{Berndt-Chan}.
However, for $n\equiv 19 \bmod 24$ the values $t_n$ are no longer class invariants
and Ramanujan 
 proposed the use of the values 
$H_n=27 t_n^{-12}$ \cite[p.\ 317]{RamNotebooks}.

In this paper, we will prove that $H_n$ values are still not class invariants
since $K(H_n)$ is a quadratic extension of the ring 
class field. 
This is clearly an obstacle for the construction of the minimal polynomials of $H_n$
values, since A. Gee and P. Stevenhagen method can no longer be applied.  
Therefore, we propose a modification of their method 
that allows us to study the case of modular functions 
which do not give class invariants and then we proceed to the study 
of  $H_n$.

We explicitly describe a method for the construction of their minimal polynomials and examine
some interesting properties of these polynomials. 
Finally, we propose the use of values $A_n=27 t_n^{-12}+t_n^{12}/27$ that are class invariants
and generate the ring  class field. Unfortunately, $A_n$ are algebraic integers which  are not units.

We also study the relation with the modular functions 
$\mathfrak{g}_0,\mathfrak{g}_1,\mathfrak{g}_2,\mathfrak{g}_3$ introduced  by A. Gee in 
 \cite[p.\ 73]{GeePhD}. 
In remark \ref{10} we see how the Ramanujan values
$A_n$ are naturally introduced as generators of the invariant ring 
$\mathbb{Q}[ \mathfrak{g}_0^{12},\mathfrak{g}_1^{12},\mathfrak{g}_2^{12},\mathfrak{g}_3^{12}]$, 
under the action of a cyclic permutation $\tau$ of order $4$.  Notice that 
$  \mathfrak{g}_0^{6}(\theta)$,$\mathfrak{g}_1^{6}(\theta)$,$\mathfrak{g}_2^{6}(\theta),$
$\mathfrak{g}_3^{6}(\theta)$ are inside  the ray class  field $H_{3,\mathcal{O}}$ and we are 
able to find their minimal polynomials over the ring class field.
We believe that this method of formalizing the search of class invariants in terms of invariant 
theory can  be applied to many other cases as well. 

This method allows us to handle the case $n\equiv 3 \bmod 24$. 
In section   \ref{3mod24} we define some new class invariants and compute their
polynomials using the methods developed in the previous sections.

Finally,  we give an example of using the $A_n$ class invariant in order to construct an elliptic 
curve over the finite field $\mathbb{F}_p$, 
\[p=2912592100297027922366637171900365067697538262949\]
of prime order \[
m=2912592100297027922366635123877214056291799441739.
               \]
{\bf Acknowledgments} We would like to thank Professor Heng Huat Chan for suggesting 
the study of 
 the Ramanujan class invariant for the $-D\equiv 19 \bmod 24$ case.
 We would also like to  thank Professor Peter Stevenhagen for making 
a lot of valuable comments on a previous version of the article and for pointing us to  
the relation between the $A_n$ and  to the   modular functions~$\mathfrak{g}_i$.
We also thank Professor Jannis Antoniadis for observing a pattern for the behavior of the 
index given in section \ref{3.1}.  
Finally, we would like to 
thank the magma algebra team for providing us with a free extension of the license to their 
system.

%
%
%
\section{Class Field Theory}
A. Gee and P. Stevenhagen provided us with a method to check whether a modular function is a class invariant. 
We will follow the notation of \cite{GeeBordeaux},\cite[chapter\ 6]{Shimura} and 
\cite{steven2}.
It is known that the modular curve $X(N)$ can be defined over $\mathbb{Q}(\zeta_N)$. Let
$\mathcal{F}_N$ be the function field of $X(N)$ over $\mathbb{Q}(\zeta_N)$, i.e. the 
field of meromorphic functions on $X(N)$ with Fourier coefficients in $\mathbb{Q}(\zeta_N)$.
Observe that  $\mathcal{F}_1=\mathbb{Q}(j)$. The automorphic function field 
$\mathcal{F}$ is defined as $\mathcal{F}=\cup_{N \geq 1} \mathcal{F}_N$.

{
For the convenience of the reader we repeat here some elements of  the adelic formulation of class 
field theory and the relation to modular functions. For more information about these 
subjects  we refer to \cite[sections\ 5.2,6.4]{Shimura} and to 
the articles of Gee-Stevenhagen \cite{GeeBordeaux},\cite{GeeStevenhagen},\cite{steven2}.
}

Fix an imaginary quadratic field $K$ and an order $\mathcal{O}=\Z[\theta]$ in $K$. 
Let $K^{\mathrm{ab}}$ be the maximal abelian extension of $K$. 
For each rational prime $p\in \Z$
 we consider  $K_p=\Q_p \otimes_{\Q} K$ and 
$\mathcal{O}_p=\Z_p \otimes_\Z\mathcal{O}$.   
We will denote by $\hat{\mathbb{Z}}=\lim_{\leftarrow n} \Z/n\Z,
\hat{\mathcal{O}}=\mathcal{O}\otimes_\Z \hat\Z=\lim_{\leftarrow n} \mathcal{O}/n\mathcal{O}
=\hat{\Z}\theta+\hat{\Z}
$ the profinite completions of the rings 
$\Z,\mathcal{O}$.  
{
Notice that $\hat{\mathcal{O}^*}=\prod_{p} \mathcal{O}^*_p$.
}
We consider the group 
\begin{eqnarray*}
J_K^f=\prod_p {}'  K^*_p 
\end{eqnarray*}
of finite id\`eles of $K$. The restricted product is taken with respect to the subgroups 
$\mathcal{O}_p^* \subset K^*_p$.  
We denote by $[\sim,K]$ the Artin map on $J_K^f$. 
There is a map $g_\theta$ which connects the two  short exact sequences:
\[
 \xymatrix{
1 \ar[r] & \mathcal{O}^* \ar[r] & \prod_{p} \mathcal{O}^*_p \ar[r]^{\!\!\!\!\![\sim,K]} \ar[d]^{g_\theta}  
& 
\mathrm{Gal}(
K^{
\mathrm{ab}
}
/K(j(\theta)
) \ar[r] & 1 
\\
1 \ar[r] & \{\pm 1\} \ar[r] & \mathrm{GL}_2(\hat{\Z}) \ar[r] & \mathrm{Gal}(\mathcal{F}/\mathcal{F}_1) \ar[r] & 1
}, 
\]
such that the image of $f(\theta)^x$ of a modular function $f$ evaluated at $\theta$  
 under the Artin symbol of $x\in \mathcal{O}^*$ 
is given by 
\begin{equation} \label{int-action}
 f(\theta)^x=f^{g_\theta(x^{-1})}(\theta).
\end{equation}
The morphism $g_\theta$ is described as follows:
Every  id\`ele $x\in \hat{\mathcal{O}}^*$ corresponds to  a $2\times 2$ matrix representing the
linear action of $\theta$ on $\hat{\Z} \theta+ \hat{\Z}$ by multiplication. If $x^2+Bx+C$ is 
the irreducible polynomial of $\theta$ then the matrix for $x=s\theta+t$ is computed to be 
\[
 g_\theta(x)=\begin{pmatrix}
              t-Bs & -Cs \\ s & t
             \end{pmatrix}.
\]

\begin{theorem}
 Let $h\in \mathcal{F}$ which  does not have a pole at $\theta$ and suppose that 
$\Q(j) \subset \Q(h)$. The function value $h(\theta)$ is a class invariant if 
and only if every element of the image $g_\theta \left(\prod_p \mathcal{O}_p^* \right)
\subset \mathrm{GL}_2(\hat{\Z})$ acts trivially on $h$.
\end{theorem}
\begin{proof}
 See \cite[Cor.\ 3]{GeeBordeaux}.
\end{proof}
Now we will consider  the non class invariant case. We have the following tower of fields:
\[
 \xymatrix{
K^{\mathrm{ab}} \ar@{-}[d]^{H_1} \ar@{-}@/_2pc/[dd]_{H={\hat{\mathcal{O}}^*}/
{\mathcal{O}^*}} \\ 
K(h(\theta)) \ar@{-}[d] \\
K\big(j(\theta)\big)
}
\]
{ Consider the open subgroup} 
\[
 \mathrm{Stab}_{\Q(h)}=\{\alpha \in \mathrm{GL}_2(\hat{\Z}): h^\alpha=h\}.
\]
The preimage $g_\theta^{-1}( \mathrm{Stab}_{\Q(h)})$ contains 
$\mathcal{O}^*=\{\pm 1\}$ and 
$g_\theta^{-1}( \mathrm{Stab}_{\Q(h)}) \subset \prod_p \mathcal{O}_p^*$. 
Notice that $h(\theta)$ is a class invariant if and only if 
$g_\theta^{-1}( \mathrm{Stab}_{\Q(h)}) = \prod_p \mathcal{O}_p^*$.
Let $H_1=\mathrm{Gal}(K^{\mathrm{ab}}/K(h(\theta)))$.
We can write   $H$ as a disjoint union of the cosets $H= \bigcup \sigma_i H_1$, and
if $h(\theta)$ is not a class invariant then there is more than one coset. 

Now we will write the Shimura reciprocity law in full generality taking into account the 
full automorphism group of the function field $\mathcal{F}$. We consider the following 
two short exact sequences, connected with  morphism 
$g_\theta: J^f_K \rightarrow \mathrm{GL}_2(A^f_\Q)$:
\begin{equation} \label{classField1}
 \xymatrix{
1 \ar[r] & K^* \ar[r] & J^f_K \ar[r]^{[\sim,K]} \ar[d]^{g_\theta}  & \mathrm{Gal}(K^{\mathrm{ab}}/K) \ar[r] & 1 \\
1 \ar[r] & \Q^* \ar[r] & \mathrm{GL}_2(A^f_\Q) \ar[r] & \mathrm{Aut}(\mathcal{F}) \ar[r] & 1.
}
\end{equation}
The map $g_\theta$ is the $\Q$-linear extension of the map $g_\theta$ given in Eq. (\ref{int-action})
which is a homomorphism $J_K^f \rightarrow   \mathrm{GL}_2(A^f_\Q)$.
The action of $z\in \mathrm{GL}_2(A^f_\Q)$ on $\mathcal{F}$ is given by writing 
$z=u\alpha$ where $u\in  \mathrm{GL}_2(\hat{\Z})$ and $\alpha\in  \mathrm{GL}_2(\Q)^+$. 
The group $\mathrm{GL}_2(\Q)^+$ consists of rational $2\times 2$ 
matrices with positive determinant and 
acts on $\mathbb{H}$ via linear fractional transformations.
Then we define $f^{u\cdot \alpha}=(f^u)^\alpha$.
For more details on this construction we refer to \cite[p.\ 6]{steven2}

The Shimura reciprocity theorem states that:
\begin{theorem}
For $h\in \mathcal{F}$ and $x\in J_K^f$ we have 
\[
 h(\theta)^{[x^{-1},K]}=h^{g_\theta(x)}(\theta).
\]
\end{theorem}
The following proposition will be useful for us
\begin{proposition} \label{refpro}
 If $\mathcal{F}/\Q(h)$ is Galois then 
\[
 h(\theta)^x=h(\theta) \Leftrightarrow h^{g_\theta(x)}=h.
\]
\end{proposition}
\begin{proof}
 See \cite[eq.\ (3.5)]{steven2}.
\end{proof}
From now on we will focus on functions $h \in \mathcal{F}$ such that 
$\mathcal{F}/\Q(h)$ is Galois. Notice that if $\Q(j) \subset \Q(h)$ then 
$\mathcal{F}/\Q(h)$ is Galois since $\mathcal{F}/\Q(j)$ is.

We have the following tower of fields:
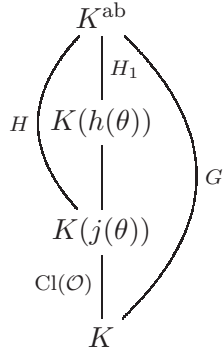
\[
 \xymatrix{
K^{\mathrm{ab}} \ar@{-}[d]^{H_1} \ar@{-}@/_2pc/[dd]_{H} \ar@{-}@/^3pc/[ddd]^{G} \\
K(h(\theta)) \ar@{-}[d] \\
K(j(\theta)) \ar@{-}[d]_{\mathrm{Cl}(\mathcal{O})} \\
K
}
\]
where $H=\mathrm{Gal}(K^{\mathrm{ab}}/K(j(\theta)))$, 
$G=\mathrm{Gal}(K^{\mathrm{ab}}/K)$, $H_1=\mathrm{Gal}(K^{\mathrm{ab}}/K(h(\theta))$
and $G/H\cong \mathrm{Cl}(\mathcal{O})$, where $\mathrm{Cl}(\mathcal{O})$ denotes the 
class group of the order $\mathcal{O}$.
We now form  the following short exact sequence:
\begin{equation} \label{ses1}
 1 \rightarrow \frac{H}{H_1} \rightarrow \frac{G}{H_1} \rightarrow \frac{G}{H} \rightarrow 1.
\end{equation}
Notice that $H/H_1 \cong \mathrm{Gal}(K(h(\theta))/K(j(\theta)))$.

Suppose that $h(\theta)$ is an algebraic integer. 
The class group of $\mathcal{O}$ is identified with  the set of primitive forms $[a,b,c]$ of discriminant $D$.
We also set $\tau_{[a,b,c]}=\frac{-b+\sqrt{d}}{2a}$. 
The proposition \ref{minpoly1}  will provide us  a method to 
compute its minimal polynomial in  $\Z[x]$.

For every element $[a,b,c] \in \mathrm{Cl}(\mathcal{O})=G/H$  we fix a 
representative $\sigma_{[a,b,c]} \in G$ such that $[a,b,c]=\sigma_{[a,b,c]} H$. Notice that 
the selection of the representative does not matter when one is acting on 
$K(j(\theta))=(K^{\mathrm{ab}})^H$ since $H$ acts trivially on $K(j(\theta))$.

The situation changes if  we try to act with $\sigma_{[a,b,c]}$ on the field $K(h(\theta))$
which is the fixed field of $H_1$ with $H_1 < H$. The class 
$\sigma_{[a,b,c]} H$ gives rise to $[H:H_1]$ classes in $G/H_1$, 
namely $\sigma_{[a,b,c]} \sigma_i H_1$, where $\sigma_1,\ldots,\sigma_s$ 
are some  coset representatives of $H_1$ in $H$ and $s=[H:H_1]$.
The action of the representative  
$\sigma_{[a,b,c]} \sigma_i=\sigma_i \sigma_{[a,b,c]}$ on $K(h(\theta))$ is 
now well defined.  Notice also that when $[a,b,c]$ runs over $G/H$ and $i$ runs over 
$1,\ldots,s$  then 
$\sigma_i \sigma_{[a,b,c]}$ runs over $G/H_1$.

\begin{proposition} \label{minpoly1}
Assume that $h(\theta) \in \mathbb{R}$ and $h(\theta)$ is algebraic. 
 Let $H_1$ be the  subgroup of $G$ that stabilizes the field $K(h(\theta))$ and let 
$H$ be the  subgroup corresponding to the ring  class field $K(j(\theta))$ of $K$. 
We consider the elements $h(\theta)^{\sigma_i\sigma_{[a,b,c]}}$. 
The polynomial
\begin{equation}\label{poldef22}
 p_{h(\theta)}:=\prod_{i=1}^s \prod_{[a,b,c] \in \mathrm{Cl}(\mathcal{O})} 
\left(
x- \left(h(\theta)^{\sigma_i \sigma_{[a,b,c]}}\right)
\right)
\end{equation}
is a polynomial in $\Z[x]$.
\end{proposition}
\begin{proof}
We have already observed that the product in eq. (\ref{poldef22}) runs over all elements in 
$\mathrm{Gal}(K(h(\theta))/K(j(\theta))$. 
We have the following tower of field extensions
\[
 \xymatrix{
& K(h(\theta)) \ar@{-}[dl]_{H/H_1} \ar@{-}[dr]^{G_1}  & \\
K(j(\theta)) \ar@{-}[d]^{\mathrm{Cl}(\mathcal{O})} \ar@{-}[drr]^{G_2} & &
 \mathbb{Q}(h(\theta)) \ar@{-}[d]^{H/H_1} \\
K \ar@{-}[dr]_{\mathrm{Gal}(K/\mathbb{Q})} & & \mathbb{Q}(j(\theta)) 
 \ar@{-}[dl]^{\mathrm{Cl}(\mathcal{O})} \\
& \mathbb{Q} & 
}
\]
where $G_1,G_2$ are lifts of $\mathrm{Gal}(K/\mathbb{Q})$.
From the diagram above we deduce that  
 $\mathrm{Gal}(\mathbb{Q}(h(\theta))/\mathbb{Q})=
\mathrm{Gal}(K(h(\theta))/K)$.
This proves that  
 the   polynomial $p_{h(\theta)}$ defined in Eq.  (\ref{poldef22}) is the defining polynomial 
of the extension $\mathbb{Q}(h(\theta))/\mathbb{Q}$. Moreover the coefficients of  $p_{h(\theta)}$
 are algebraic integers in $\Q$ 
therefore $p_{h(\theta)}\in \Z[x]$. 
\end{proof}

\begin{remark}
 The assumption $h(\theta)\in \mathbb{R}$ is essential as one 
sees in section \ref{3mod24}, where we compute the minimal 
polynomial of the class invariant $\mathfrak{g}_2^6(\theta)$.
\end{remark}

The above construction becomes practical if $h \in \mathcal{F}_N$ is a 
modular function of level $N$. Then the  value $h(\theta)$ is known to be
inside the ray class field modulo $N$ and the action of $\hat{\mathcal{O}}^*$ 
can be computed in terms of a finite quotient $(\mathcal{O}/N \mathcal{O})^*$.
Here it is important to assume also that $\Q(j) \subset \Q(h)$ so proposition \ref{refpro}
is applicable.
More precisely we can replace Eq.  (\ref{classField1}) with the exact sequence:
\[
 \xymatrix{
\mathcal{O}^* \ar[r] & (\mathcal{O}/N \mathcal{O})^* \ar[r]  \ar[d]_{\bar{g}_\theta} &
 \mathrm{Gal}\big(K(\mathcal{F}_N(\theta))/K(j(\theta))\big) \ar[r] & 1\\
\{\pm 1\} \ar[r] & \mathrm{GL}_2(\Z/N\Z) \ar[r] & \mathrm{Gal}(\mathcal{F}_N/\mathcal{F}_1) \ar[r] & 1,
}
\]
where we have considered the reduction of all rings and maps  modulo $N$.
The strategy for the computations is the following:
compute generators $x_1,\ldots,x_k$ for the group $(\mathcal{O}/N \mathcal{O})^*$ and map them to 
$\mathrm{GL}_2(\Z/N\Z)$ using $\bar{g}_\theta$. If each matrix $g(x_i)\in \mathrm{GL}_2(\Z/N\Z)$ 
acts trivially on $h$ then $h(\theta)$ is a class invariant. If not we can consider the subgroup 
$A \subset \bar{g}_\theta \left((\mathcal{O}/N \mathcal{O})^* \right)$ that 
acts trivially on $h$. The Galois group of $K(h(\theta))/K(j(\theta))$ equals 
\[
 \mathrm{Gal}(K(h(\theta))/K(j(\theta))=
\frac{(\mathcal{O}/N \mathcal{O})^*/\mathcal{O}^*}{\bar{g}_\theta^{-1} (A)}.
\]
We will now give an applicable approach to proposition \ref{minpoly1} by working modulo $N$. 
Following the article of A. Gee \cite[Eq.\  17]{GeeBordeaux} we give the next  definition.
This will allow us  to compute the action of the images of generators of 
$G_{72}$ on the modular functions of level $72$. 
\begin{definition}
Let $N \in \mathbb{N}$ and 
  $[a,b,c]$ be a representative of the equivalence class of an element in the class group.
Let $p$ be a prime number and  $p^r$ be the maximum power of $p$  that divides $N$.
Assume that the discriminant $D=b^2-4ac \equiv 1 \bmod 4$.
The following matrix definition is motivated by the explicit 
writing of the id\`ele that  locally generates $[a,b,c]$ for all primes 
$p$, see \cite[sec.\ 4]{GeeStevenhagen}.
Define the matrix
\[
 A_{[a,b,c],p^r}=
\left\{
\begin{array}{ll}
\left(
\begin{array}{cc}
 a & \frac{b-1}{2} \\
0 & 1 
\end{array}
\right)
& 
\mbox{ if } p\nmid a \\
\left(
\begin{array}{cc}
 \frac{-b-1}{2} & -c \\
1 & 0 
\end{array}
\right)
& 
\mbox{ if } p\mid a \mbox{ and } p\nmid c\\
\left(
\begin{array}{cc}
 \frac{-b-1}{2}-a  & \frac{1-b}{2}-c  \\
1 & -1 
\end{array}
\right)
& 
\mbox{ if } p\mid a \mbox{ and } p \mid c.

\end{array}
\right.
\]
The Chinese remainder theorem implies that 
\[
 \mathrm{GL}_2(\mathbb{Z}/N\mathbb{Z}) \cong \prod_{p \mid N} 
\mathrm{GL}_2(\mathbb{Z}/p^r \mathbb{Z}).
\]
We define $A_{[a,b,c]}$ as  the unique element in $
\mathrm{GL}_2(\mathbb{Z}/N\mathbb{Z})$ that it is mapped 
to $A_{[a,b,c],p^r}$  modulo $p^r$ for all $p\mid N$.
This matrix $A_{[a,b,c]}$ can be written  uniquely as a product 
\begin{equation} \label{prodexp}
A_{[a,b,c]}=
B_{[a,b,c]}
\left( 
\begin{array}{cc}
1 & 0 \\
0 & d_{[a,b,c]}
\end{array}
\right),
\end{equation}
where $d_{[a,b,c]}=\det A_{[a,b,c]}$ and $B_{[a,b,c]}$ is a matrix with determinant $1$.
We will denote by $\sigma_{d_{[a,b,c]}}$ the automorphism of $\mathbb{Q}(\zeta_N)$ 
sending $\zeta_N \mapsto \zeta_N^{d_{[a,b,c]}}$. 
\end{definition}

Let $\lambda \in \Q(\zeta_{N})$.
Shimura reciprocity law gives us \cite[lemma.\ 20]{GeeBordeaux}  the action of 
 $[a,b,c] $    on $\lambda h(\theta)$ for $\theta=1/2+i \sqrt{n}/{2}$: 
\[
  \big( \lambda h(\theta)\big) ^{[a,-b,c]}= \lambda^{\sigma_{d_{[a,b,c]}}} 
h\left( 
\frac{\alpha_{[a,b,c]} \tau_{[a,b,c]}+ \beta_{[a,b,c]} }
{\gamma_{[a,b,c]} \tau_{[a,b,c]} + \delta_{[a,b,c]} }
\right)^{\sigma_{d_{[a,b,c]}}},
\]
where
 $\begin{pmatrix} \alpha_{[a,b,c]} & \beta_{[a,b,c]} \\ \gamma_{[a,b,c]} & \delta_{[a,b,c]} \end{pmatrix}=A_{[a,b,c]}$
and $\tau_{[a,b,c]}$ is the (complex) root of $az^2 +bz+c$ with positive imaginary part.

%
 
\begin{theorem} \label{shimura1}
Let $\mathcal{O}=\mathbb{Z}[\theta]$ be an order  of the imaginary quadratic field $K$,
and assume that $x^2+Bx+C$ is the minimal polynomial of $\theta$.
Let $N>1$ be a natural number, $x_1,\ldots,x_r$ be  generators of the abelian group $\left(\mathcal{O}/N \mathcal{O} \right)^*$ and
$\alpha_i+\beta_i\theta \in \mathcal{O}$ be a representative of the class of the generator $x_i$. 
For each representative we consider 
 the matrix:
\[A_i:=\begin{pmatrix} \alpha_i-B\beta_i & -C \beta_i \\ \beta_i & \alpha_i \end{pmatrix}.\]
If $f$ is a modular function of level $N$ and if  for all matrices $A_i$ it holds that
\begin{equation} \label{act123}
f(\theta)=f^{A_i}(\theta), \mbox{ and } \mathbb{Q}(j) \subset \mathbb{Q}(f)
\end{equation}
then $f(\theta)$ is a class invariant.
\end{theorem}
\begin{proof}
\cite[Cor.\ 4]{GeeBordeaux} for the maximal order case and \cite[section\ 5]{steven2}
for the general case. 
\end{proof}

 
%
%
%
\section{Ramanujan Invariants} \label{sec3}
We would like to find the minimal polynomial in $\Z[x]$ of the Ramanujan invariants 
$H_n=27/t_n^{12}$ for values $n\equiv 19 \bmod 24$.
In \cite{KoKo} the authors introduced the modular functions 
$R,R_1,\ldots, R_5$ of level $N=72$ in order to study $t_n$.
P. Stevenhagen pointed to us that the functions $R_i$ can be expressed 
in terms of the generalized Weber functions $\mathfrak{g}_0,\mathfrak{g}_1,\mathfrak{g}_2,\mathfrak{g}_3$
 defined in the 
 work of A. Gee in \cite[p.\ 73]{GeePhD}
as 
\[
 \mathfrak{g_0}(\tau)=\frac{\eta(\frac{\tau}{3})}{\eta(\tau)},\;
 \mathfrak{g_1}(\tau)=\zeta_{24}^{-1}\frac{\eta(\frac{\tau+1}{3})}{\eta(\tau)},\;
\mathfrak{g_2}(\tau)=\frac{\eta(\frac{\tau+2}{3})}{\eta(\tau)},\;
\mathfrak{g_3}(\tau)=\sqrt{3}\frac{\eta(3\tau)}{\eta(\tau)},
\]
where $\eta$ denotes the Dedekind eta function:
\[
 \eta(\tau)=e^{2\pi i \tau/24} \prod_{n\geq 1}(1-q^n)\;\; \tau \in \mathbb{H}, q=e^{2\pi i \tau}.
\]
\begin{proposition} \label{fund-ref}
 The functions $\frak{g}_i^{12}$ satisfy the polynomial:
\[
 X^4+36 X^3+270 X^2+(756-j) X+3^6-0.
\]
In particular $\Q(h)\subset \Q(\frak{g}_i)$ and $\mathcal{F}/\Q(\frak{g}_i)$ is Galois.
\end{proposition}
\begin{proof}
This is a classical result see
\cite[eq.\ 5 p.\ 73]{GeePhD}, \cite[p.\ 255]{Weber}.
\end{proof}

Here  will need only the $R_2(\tau)$ and $R_4(\tau)$ defined by: 
\begin{equation} \label{R2}
R_2(\tau)= \frac{\eta(3\tau)\eta(\tau/3+2/3)}{ \eta^2(\tau)}=
\sqrt{3}^{-1}\mathfrak{g}_2(\tau)\mathfrak{g}_3(\tau)
\end{equation}

\begin{equation} \label{R4}
R_4(\tau)= \frac{\eta(\tau/3)\eta(\tau/3+1/3)}{ \eta^2(\tau)}=
\zeta_{24}\mathfrak{g}_0(\tau)\mathfrak{g}_1(\tau).
\end{equation}
The six modular functions $R_i$ defined in 
 \cite{KoKo} 
correspond to the $\binom{4}{2}=6$ different products
$\mathfrak{g}_i\mathfrak{g}_j$ we can make from $\mathfrak{g}_i,$ $i=0,\ldots,3$.

The Ramanujan value can be expressed in terms of the above modular functions as  
\begin{equation} \label{gg}
 t_n=\sqrt{3}R_2\left( -\frac{1}{2}+i \frac{\sqrt{n}}{2} \right)=(\mathfrak{g}_2\mathfrak{g}_3)
\left(  -\frac{1}{2}+i \frac{\sqrt{n}}{2}\right).
\end{equation}
Notice also that $\sqrt{3}=\zeta_{72}^6-\zeta_{72}^{30}$.
The Ramanujan invariants for $D \equiv 5 \bmod 24$ are
\[
 H_n:=\frac{27}{t_n^{12}}
\]
 and we also define the values
\[
 A_n:=H_n+\frac{1}{H_n}=\frac{27}{t_n^{12}}+\frac{t_n^{12}}{27}.
\]

Denote by $S$ the involution $\tau \mapsto -\frac{1}{\tau}$ and by $T$ the map 
$\tau \mapsto \tau+1$. The elements $S,T$ generate the group $\mathrm{SL}(2,\Z)$. 
We will use the following
\begin{lemma}
 The action of  $S:z\mapsto -1/z$ on $\mathfrak{g}_i$ is given by 
\[
(\mathfrak{g}_0,\mathfrak{g}_1,\mathfrak{g}_2,\mathfrak{g}_3) 
 \begin{pmatrix}
  0 &  0          & 0  & 1 \\
  0 &  0          & \zeta_{72}^{6} &0 \\
  0 &  \zeta_{72}^{-6}  & 0 & 0\\
  1 &  0          & 0 & 0\\
 \end{pmatrix}
\]
and the action of $T:z \mapsto z+1$ on $\mathfrak{g}_i$  is given by
\[
(\mathfrak{g}_0,\mathfrak{g}_1,\mathfrak{g}_2,\mathfrak{g}_3) 
 \begin{pmatrix}
  0 &  0          & 1 & 0 \\
  1 &  0          & 0 &0 \\
  0 &  \zeta_{72}^{-6} & 0 & 0\\
  0 &  0          & 0 & \zeta_{72}^6\\
 \end{pmatrix}.
\]
The action of $\sigma_d$ on $\mathfrak{g}_i$ is given in terms of the following matrix
\[
 (\mathfrak{g}_0,\mathfrak{g}_1,\mathfrak{g}_2,\mathfrak{g}_3) 
 \begin{pmatrix}
  1 &  0          & 0 & 0 \\
  0 &     \zeta_{72}^{-2d+2}       & 0 &0 \\
  0 &  0            & \zeta_{72}^{2d-2} & 0\\
  0 &  0          & 0 & \frac{ \zeta_{72}^{6d}-\zeta_{72}^{30d}}
{\zeta_{72}^6 -\zeta_{72}^{30}}\\
 \end{pmatrix} \mbox{ if } d\equiv 1 \bmod 3 
\]
and 
\[
 (\mathfrak{g}_0,\mathfrak{g}_1,\mathfrak{g}_2,\mathfrak{g}_3) 
 \begin{pmatrix}
  1 &  0          & 0 & 0 \\
  0 &  0          & \zeta_{72}^{2d+2} &0 \\
  0 &  \zeta_{72}^{-2d-2}            & 0 & 0\\
  0 &  0          & 0 & 
\frac{ \zeta_{72}^{6d}-\zeta_{72}^{30d}}{\zeta_{72}^6 -\zeta_{72}^{30}}\\
 \end{pmatrix} \mbox{ if } d\equiv 2 \bmod 3 
\]
\end{lemma}
\begin{proof}
The action of $S,T$ follows by using the  transformation formulas of the $\eta$-function 
\cite{SilvII}:
\[
\eta(\tau+1)=e^{2\pi i\tau/24}\eta(\tau)\mbox{ and }
\eta\left(-\frac{1}{\tau}\right)=\sqrt{-i\tau}\eta(\tau).\]
For the action of $\sigma_d$ observe for example that 
\begin{eqnarray*}
 \eta \left(\frac{\tau}{3}+\frac{1}{3} \right) &= &
\exp\left (\frac{2\pi i }{24} \left(\frac{\tau}{3}+\frac{1}{3} \right)\right) 
\sum_{\nu=0}^{\infty} a_\nu
 \exp \left( 
\frac{2 \pi i \nu}{3 }\tau + \frac{2 \pi i \nu }{3}
\right)=\\
&=&\exp\left (\frac{2\pi i }{24} \frac{\tau}{3}\right) \zeta_{72} \sum_{\nu=0}^{\infty} \zeta_3^{\nu} a_\nu
 \exp \left( 
\frac{2 \pi i \nu}{3 }\tau 
\right).
\end{eqnarray*}
The element $\sigma_d :\zeta_{72} \mapsto \zeta_{72}^d$ sends $\zeta_3^\nu$ to 
$\zeta_3^{d\nu}=\zeta_3$  if $d\equiv 1 \bmod 3$ and to 
$\zeta_3^{2\nu}$ if $d\equiv 2 \bmod 3$.
Therefore 
\[
\sigma_d(\mathfrak{g}_1(\tau))=\sigma_d
\left(\zeta_{24}^{-1} 
\frac{\eta(\frac{\tau+1}3)}{\eta(\tau)}
\right)=
\left\{
\begin{array}{ll}
 \zeta_{72}^{-2d+2} \mathfrak{g}_1 & \mbox{ if } d\equiv 1 \bmod 3 \\
 \zeta_{72}^{-2d-2} \mathfrak{g}_2 & \mbox{ if } d\equiv 2\bmod 3 
\end{array}
\right. 
\]
\end{proof}
\begin{remark}
We have a representation 
\[
 \rho:\mathrm{SL}(2,\Z) \rightarrow \langle \mathfrak{g}_0,\mathfrak{g}_1,\mathfrak{g}_2,\mathfrak{g}_3
 \rangle_\mathbb{R}=V.
\]
This representation gives rise to the representation $\mathrm{Sym}^2 V$, where 
the space $\mathrm{Sym}^2 V$ has dimension $\binom{4}{2}=6$ and it is 
generated by the elements  $\mathfrak{g}_i \mathfrak{g}_j$, $1\leq i <j \leq 4$. The representation 
$\mathrm{Sym}^2 V$ was an alternative way to express the action given in \cite{KoKo}
in terms of the modular functions $R_i$. 
\end{remark}

We first study the group $(\mathcal{O}/72 \mathcal{O})^*$ for the values 
$n=19,43,67$.
Using Chinese remainder theorem we compute first that 
\[
 \left( \frac{\mathcal{O} }{72 \mathcal{O} } \right)^* \cong 
\left(\frac{\mathcal{O}}{9\mathcal{O}}\right)^* \times \left(\frac{\mathcal{O}}{8\mathcal{O}}\right)^*.
\]
Notice that the assumptions we put force $2\mathcal{O},3\mathcal{O}$ to be prime ideals.
The structure of the group $\left(\frac{\mathcal{O} }{P^k \mathcal{O} } \right)^*$ for a prime 
ideal of $\mathcal{O}$ is given by the following:
\begin{theorem} \label{cohen}
 Let $P$ be a prime ideal of $\mathcal{O}$ of inertia degree $f$ over the field of rationals, i.e.
if $p$ is the generator of the principal ideal $P\cap \Z$ then $N(P)=p^f$ and assume that 
the ramification index $e(P/p)=1$.
The group $\left(\frac{\mathcal{O} }{P^k \mathcal{O} } \right)^*$ is isomorphic to 
the direct product $\left(\frac{\mathcal{O} }{P \mathcal{O} } \right)^* \times \frac{1+P}{1+P^k}$. 
The group $\left(\frac{\mathcal{O} }{P \mathcal{O} } \right)^*$ is cyclic of order $p^f-1$.
If $p\geq \min\{3,k\}$ then the group 
$ \frac{1+P}{1+P^k}$ is isomorphic to $\left(\frac{\Z}{p^{k-1} \Z}\right)^f$. If 
$p=2$ and $k=3$ 
then  $ \frac{1+P}{1+P^3}$ is isomorphic to $\left(\frac{\Z}{2\Z}\right)^2 \times \left(\frac{\Z}{4\Z}\right)^{f-1}$. 
\end{theorem}
\begin{proof}
The group $\left(\frac{\mathcal{O} }{P^k \mathcal{O} } \right)^*$ is isomorphic to 
the direct product $\left(\frac{\mathcal{O} }{P \mathcal{O} } \right)^* \times \frac{1+P}{1+P^k}$ by 
proposition \cite[prop.\ 4.2.4]{CohenAdv}. 
For $e(P/p)=1$ and $p\geq \min\{3,k\}$ the  $P$-adic logarithmic function defines an isomorphism of 
 the mutliplicative group  $\frac{1+P}{1+P^k}$ to the additive group $P/P^k$
which in turn is isomorphic to $\mathcal{O}/P^{k-1}$ by \cite[lemma\ 4.2.9]{CohenAdv}.
The condition $p \geq  \min\{3,k\}$ is put  so that the logarithmic function converges.
By \cite[th.\ 4.2.10]{CohenAdv} we have 
\[
 \frac{\mathcal{O}}{P^{k-1}}\cong
\left(\frac{\Z}{p^q \Z}\right)^{(r+1)f} \times 
\left(
\frac{\Z}{p^{q-1}\Z}
\right)^{(e-r-1)f}
\]
where $k+e-2=eq+r$, $0\leq r <e$. If $e=1$ then the last formula becomes:
\[
 \frac{\mathcal{O}}{P^{k-1}}\cong
\left(\frac{\Z}{p^{k-1} \Z}\right)^{f}. 
\]
The case $p=2$ and $k=3$ is studied in \cite[prop.\ 4.2.12]{CohenAdv}.
\end{proof}
By applying theorem \ref{cohen}
we find the structure of the multiplicative groups
\[
 \left(\frac{\mathcal{O}}{9\mathcal{O}}\right)^*\cong \frac{\Z}{8\Z} \times \frac{\Z}{3\Z} \cong
\frac{\Z}{24\Z} \times \frac{\Z}{3\Z}
\]
and
\[
 \left(\frac{\mathcal{O}}{8\mathcal{O}}\right)^* \cong \frac{\Z}{3\Z} \times 
\left(\frac{\Z}{2\Z}\right)^2 \times \frac{\Z}{4\Z} \cong \frac{\Z}{12\Z} \times \left(\frac{\Z}{2\Z}\right)^2
\]
For finding the generators of these groups one can use the $P$-adic logarithmic function in order to 
pass from the multiplicative group $\frac{1+P}{1+P^k}$ to the 
additive group $\mathcal{O}/P^{k-1}\mathcal{O}$. This method does work only  for large primes 
(so that the logarithmic function is convergent) and not for the case $p=2$, $k=3$. 

In order to find the generators we proceed as follows: We exhaust all units in $\mathcal{O}/9\mathcal{O}$ 
until we find one unit $U_1$ of order $24$ then we remove this unit and all its powers from the set of possible units 
and we try again in order to find a unit $U_2$ of order $3$. For the units in $\mathcal{O}/8\mathcal{O}$
we work similarly. We first find  a unit  $V_1$ of   maximal order  $12$ remove all its powers 
 from the set of units and we  try again
in order to find a unit  $V_2$ of order $2$. We  remove   all products of powers of $U_1$ and $U_2$
 and then we search on the remaining units for the third generator $V_3$.
 Finally  we lift these  units to units of the ring $ \mathcal{O}/72\mathcal{O}$
using the Chinese remainder theorem.
%
%
%
This way we arrived to  the following generators of the group $(\mathcal{O}/72 \mathcal{O})^*$:
  $5\theta+7$, $6\theta+7$, $7\theta +7$, $4\theta+7$, $4\theta+1$.
The orders of the generators of the group $(\mathcal{O}/72 \mathcal{O})^*$ are given in the 
following table:
\begin{center}
\begin{tabular}{|c|c|c|c|c|c|}
\hline
 Generator & $5\theta+7$ & $6\theta+7$ & $7\theta +7$ & $4\theta+7$ & $4\theta+1$ \\
\hline
Order        & 24               &   3 & 12 & 2 & 2\\
\hline
\end{tabular}
\end{center}
These generators will be mapped to matrices $A_i$ defined in theorem \ref{shimura1}.

For example the generator $5 \theta+7$ in $(\mathcal{O}/9\mathcal{O})^*$ corresponds to the matrix
\[
\begin{pmatrix} 3 & 8 \\5 & 16 \end{pmatrix}= \begin{pmatrix} 3 & 1 \\ 5 & 2 \end{pmatrix} \begin{pmatrix} 1 & 0 \\ 0  & 8 \end{pmatrix},
\]
where $M=\begin{pmatrix} 3 & 1 \\ 5 & 2 \end{pmatrix}$ is a matrix of determinant $1 \bmod 9$. 
Let 
\[T=\begin{pmatrix} 1 & 1 \\0 & 1 \end{pmatrix} \mbox{ and } S=\begin{pmatrix} 0 & 1 \\ -1 & 0 \end{pmatrix}.\]
The matrix $M$ can be decomposed according to \cite{GeeStevenhagen} as 
$\bar{T}_9^8 \bar{S}_9 \bar{T}_9^5 \bar{S}_9 \bar{T}_9^6$ where  
\[\bar{S}_9=T^{-1} ST^{-65} ST^{-1} ST^{1096} \mbox{ and } \bar{T}_9=T^{-9} \] according to \cite{KoKo}.
The action of the generators on the elements $\mathfrak{g}_0,\mathfrak{g}_1,\mathfrak{g}_2,\mathfrak{g}_3$
is computed by magma and it is given in  table \ref{Table:orders}.

\begin{table}\caption{Orders and generators of the group $(\mathcal{O}/72 \mathcal{O})^*$}
\label{Table:orders}
\begin{tabular}{|c|r|r|r|r|r|}
\hline
                     & $5\theta+7$ & $6\theta+7$ & $7\theta +7$ & $4\theta+7$ & $4\theta+1$ \\
\hline
$\mathfrak{g}_0$        &  $(-\zeta_{72}^{18} + \zeta_{72}^6)\mathfrak{g}_2$              &  $\zeta_3\mathfrak{g}_0$ & $\mathfrak{g}_0$ &  $-\mathfrak{g}_0$ &$\mathfrak{g}_0$ \\
\hline
$\mathfrak{g}_1$        &  $\zeta_{72}^{12}\mathfrak{g}_3$             &$\zeta_3\mathfrak{g}_1$   & $-\mathfrak{g}_1$ & $-\mathfrak{g}_1$  & $\mathfrak{g}_1$  \\
\hline
$\mathfrak{g}_2$         &  $ -\mathfrak{g}_1$              &  $ -\zeta_{72}^{12}\mathfrak{g}_2$ & $-\mathfrak{g}_2$ &$ -\mathfrak{g}_2$ & $\mathfrak{g}_2$ \\
\hline
$\mathfrak{g}_3$        &   $ (-\zeta_{72}^{18} + \zeta_{72}^6)\mathfrak{g}_0$          & $ -\zeta_{72}^{12}\mathfrak{g}_3 $ & $ -\mathfrak{g}_3$ & $ -\mathfrak{g}_3$   & $\mathfrak{g}_3$ \\
\hline
\end{tabular}
\end{table}
\begin{lemma} 
 The quantities $\mathfrak{g}_i(\theta)^6$ are in the ray  class field of conductor $3$.   
\end{lemma}
\begin{proof}
There is the following diagram with exact rows for every $N$ (here we will use the values $N=72,3$:
\[
 \xymatrix{
\mathcal{O}^* \ar[r] & (\mathcal{O}/N \mathcal{O})^* \ar[r] \ar^{\bar{g}_\theta}[d]  & 
\mathrm{Gal}(H_{N,\mathcal{O}} /H_\mathcal{O}) \ar[r] & 1 \\
\{\pm 1\} \ar[r] & 
\mathrm{GL}_2(\Z/N\Z) \ar[r] & \mathrm{Gal}(\mathcal{F}_N/\mathcal{F}_1) \ar[r] &1,
}
\]
where $H_{N,\mathcal{O}}$ denotes the ray class field of conductor $N$. The epimorphism 
of the upper row is induced by the Artin map and allows us to see 
elements in $(\mathcal{O}/N \mathcal{O})^*$ as elements in 
$\mathrm{Gal}(H_{N,\mathcal{O}} /H_\mathcal{O})$.

 The ray class field $H_{3,\mathcal{O}}$  of conductor $3$ is an extension of degree $4$ 
of the ring  class field, 
as one computes looking at $(\mathcal{O}/3 \mathcal{O})^*/\mathcal{O}^*$.
Indeed, the group $(\mathcal{O}/3 \mathcal{O})^*$  is isomorphic to a cyclic group of order $8$ by theorem 
\ref{cohen} and by taking the quotient of $\mathcal{O}^*=\{\pm 1\}$ we arive at a group of order $4$.
 
The element $5 \theta+7$ generates a subgroup of order $24$ in 
$(\mathcal{O}/72 \mathcal{O})^*$. 
This means that the ray class field of conductor $3$ is the fixed field 
of $\langle (5 \theta+7)^4 \rangle$ and all other generators 
$ 6\theta+7,7\theta +7,4\theta+7,4\theta+1$. 

We compute that the action of $(5 \theta+7)^4=3\theta +8$ on 
$\mathfrak{g}_0,\mathfrak{g}_1,\mathfrak{g}_2,\mathfrak{g}_3$ is given by 
\[
 (\mathfrak{g}_0,\mathfrak{g}_1,\mathfrak{g}_2,\mathfrak{g}_3)\mapsto 
\left( (-\zeta_{72}^{12}+1) \mathfrak{g}_0,
 (-\zeta_{72}^{12}+1)\mathfrak{g}_1,
\zeta_{72}^{12} \mathfrak{g}_2,
\zeta_{72}^{12} \mathfrak{g}_3
\right)
\]

Since $\zeta_{72}^{12}, (\zeta_{72}^{12} - 1)$ are $6$th roots of unity 
we see that $(5 \theta+7)^4$, indeed leaves $\mathfrak{g}_0^6,\mathfrak{g}_1^6,
\mathfrak{g}_2^6,\mathfrak{g}_3^6$
invariant. 
On the other hand,  looking at table \ref{Table:orders} we see that 
 all other generators leave also $\mathfrak{g}_0^6,\mathfrak{g}_1^6,\mathfrak{g}_2^6,\mathfrak{g}_3^6$
 invariant.

Notice that the Galois group $\mathrm{Gal}(H_{n,\mathcal{O}}/K)$ is cyclic of order $4$ generated by 
$ 5\theta+7$ and the action is given by 
\[
 (\mathfrak{g}_0^6,\mathfrak{g}_1^6,\mathfrak{g}_2^6,\mathfrak{g}_3^6)\mapsto 
( -\mathfrak{g}_2^6,
\mathfrak{g}_3^6,
-\mathfrak{g}_1^6,
-\mathfrak{g}_0^6).
\]
\end{proof}
\begin{remark} \label{invarianttheory}
 Notice that we have a polynomial action of the 
permutation group $\langle (0,2,1,3)\rangle$ \footnote{Here in order to be compatible 
with the enumeration of $\mathfrak{g}_0,\mathfrak{g}_1,\mathfrak{g}_2,\mathfrak{g}_3$ we allow $0$
as a number in the permutations.} on the polynomial ring 
$\mathbb{Q}[\mathfrak{g}_0^{12},\mathfrak{g}_1^{12},\mathfrak{g}_2^{12},\mathfrak{g}_3^{12}]$.
The ring of invariants of this action can be computed to be the 
polynomial ring generated by the polynomials
\[
\mathfrak{g}_0^{12} + \mathfrak{g}_1^{12} + \mathfrak{g}_2^{12} +\mathfrak{g}_3^{12},
   \mathfrak{g}_0^{24} + \mathfrak{g}_1^{24} + \mathfrak{g}_2^{24} + \mathfrak{g}_3^{24},
    \mathfrak{g}_0^{12}\mathfrak{g}_1^{12} + \mathfrak{g}_2^{12}\mathfrak{g}_3^{12},
    \mathfrak{g}_0^{48} + \mathfrak{g}_1^{48} + \mathfrak{g}_2^{48} + \mathfrak{g}_3^{48}.
\]
Of course $\mathfrak{g}_0^{12} + \mathfrak{g}_1^{12} + \mathfrak{g}_2^{12} +\mathfrak{g}_3^{12}=-36$ is 
an invariant of the linear action but not an interesting one.
All these (and their combinations) will give class invariants. Notice that the class invariant 
$A_n$ introduced later  in this paper comes from  the third one 
$\mathfrak{g}_0^{12}\mathfrak{g}_1^{12} + \mathfrak{g}_2^{12}\mathfrak{g}_3^{12}$. 
\end{remark}
\begin{remark} \label{10}
Every polynomial expression given in remark \ref{invarianttheory} gives rise to 
a class invariant. What are the relations of these class invariants?
Set $Y_i=\mathfrak{g}_i^{12}$. We know that $Y_i$ satisfy equation 
\begin{equation} \label{gequ}
 Y_i^4 + 36 Y_i^3+270Y_i^2+(756-j)Y_i+3^6=0, 
\end{equation}
by proposition \ref{fund-ref}.
The first invariant given in remark \ref{invarianttheory} is just $36$. We then have 
\[
 36^2=\left( \sum_{i=0}^3 Y_i \right)^2=\left( \sum_{i=0}^3 Y_i^2 \right)+
2 \sum_{0\leq i<j \leq 3} Y_i Y_j = \left( \sum_{i=0}^3 Y_i^2 \right)-540.
\]
Therefore, 
\[
 \sum_{i=0}^3 Y_i^2=36^2+540=1836.
\]
We compute that 
\begin{eqnarray}
 36^3=\left( \sum_{i=0}^3 Y_i \right)^3 & = &\sum_{i=0}^3 Y_i ^3+ 6 \sum_{i,j,k} Y_i Y_j Y_k +
3 \sum_{i\neq j} Y_i Y_j^2 \nonumber \\
& =& \sum_{i=0}^3 Y_i ^3 +6(756-j)+ 3 \sum_{i\neq j} Y_i Y_j^2 \label{mac1}.
\end{eqnarray}
We now compute
\begin{eqnarray} \label{mac2}
1836\cdot 36 & =&  \left(\sum_{i=0}^3 Y_i  \right)\left(\sum_{j=0}^3 Y_i^2  \right)=  \sum_{i=0}^3 Y_i ^3+\sum_{i\neq j} Y_i Y_j^2.
\end{eqnarray}
By combining eq. (\ref{mac1}) and (\ref{mac2}) we obtain:
\[
 2 \sum_{i=0}^3 Y_i^3=151632+6(756-j).
\]
Finally, we compute that the last invariant given in remark \ref{invarianttheory} is given by eq. 
(\ref{gequ})
\[
 \sum_{i=0}^3 Y_i^4=-36 \sum_{i=0}^3 Y_i^3-270\sum_{i=0}^3Y_i^2-(756-j)\sum_{i=0}^3
Y_i-3^6.
\]
This means that the invariants of  remark \ref{invarianttheory}
are either constant or linear combinations of $j$ (and these would 
give polynomials with the same growth as the Hilbert polynomials)
and  
 $Y_0Y_1+Y_2Y_3$ which gives by evaluation at $\theta$ the 
$A_n$ class invariants. 
\end{remark}

\begin{remark}
 Notice that equation (\ref{gequ}) allows us to find the minimal polynomials (over the ring 
class field) of  the quantities 
$Z_i:=\mathfrak{g}_i(\theta)^6$, just by replacing $Y_i$ by $Z_i^2$.  
\end{remark}

\begin{remark}
 Notice that using only powers of the $\mathfrak{g}_i$ modular functions 
we can only construct an extension of the ring  class field of order $4$. 
The Ramanujan invariants $H_n$ allow us to construct a quadratic extension of the 
ring  class field.
\end{remark}

We return now to the study of Ramanujan invariants. 
Using magma and the above computations  we compute that 
$5 \theta+7$ sends  $(1/R_2)^{12}$  to $-3^6/R_4^{12}$.
 Therefore  
$H_n$ is not a class invariant. 
Similarly we compute that all other generators of 
$\left( \frac{\mathcal{O} }{72 \mathcal{O} } \right)^*$
act trivially on $(1/R_2)^{12}$.
 The field generated by the class invariant $H_n$ is a quadratic extension 
of the ring  class field of $K$.

On the other hand the above computation allows us to compute the minimal polynomial $p_n\in \Z[x]$ of $H(n)$
by using the formula
\begin{equation} \label{pn}
 p_n (x) =\prod_{[a,b,c] \in \mathrm{Cl}(\mathcal{O})} 
\left( x-  3^{-3} R_2(\tau_0)^{-12 [a,-b,c]} \right) 
\left( x+  3^3 R_4(\tau_0)^{-12 [a,-b,c]} \right),
\end{equation}
with $\tau_0=\frac{-1+i\sqrt{n}}{2}$.
The results of these  computations for some values  $n=19+24i$, $i=0,\ldots 18$
 are shown in Table \ref{pnpoly1}.
We will now prove some properties for the minimal polynomials.
We will need the following lemma.
\begin{lemma} \label{invertibleR2}
The following identity holds:
\[
 \left( R_2(\tau)R_4(\tau) \right)^{12}=-1.
\]
\end{lemma}
 \begin{proof}
A. Gee in \cite[p.\ 73]{GeePhD} observes that $\mathfrak{g}_0\mathfrak{g}_1\mathfrak{g}_2\mathfrak{g}_3=\sqrt{3}$.
The result follows by   eq. (\ref{R2}), (\ref{R4}).
\end{proof}

\begin{lemma} \label{invert111}
 Consider  a monic  polynomial $f(x)=x^n+\sum_{\nu=0}^{n-1} a_\nu x^{\nu}$, with $n$ even.
Consider the set of roots $\Sigma=\{\rho_1,\ldots,\rho_n\}$ of $f$ and assume that 
$f$ has no multiple roots. If the transformation 
$x \mapsto 1/x$ sends the above defined set of roots $\Sigma$ to $\Sigma$ then $a_0=1$ and $a_{\nu}=a_{n-\nu}$.
\end{lemma}
\begin{proof}
Write $f=\prod_{i=1}^n(x-\rho_i)$. By the assumption all roots $\rho_i\neq 0$. 
The result follows from the fact that the ``reverse polynomial''
$x^nf(1/x)$  is the polynomial $\prod_{i=1}^n (1-\rho_i X)$ having 
the reciprocals of $\rho_i$ as roots. 
\end{proof}
\begin{proposition}
 The minimal polynomials $p_n(x)=x^{2h} +\sum_{\nu=0}^{2h-1} a_\nu x^\nu$ of $H(n)$ are palindromic, i.e. 
$a_\nu=a_{2h-\nu}$. The constant coefficient $a_0$ equals $1$. 
\end{proposition}
\begin{proof}
From Eq.  (\ref{pn}) we have that 
whenever 
\[
H(n)^{[a,-b,c]}=3^{-3} R_2(\tau_0)^{-12 [a,-b,c]} 
\]
is a root then 
$
- 3^3 R_4(\tau_0)^{-12 [a,-b,c]}
$ is a root. But lemma \ref{invertibleR2} implies that 
\[
-3^3 R_4(\tau_0)^{-12 [a,-b,c]}=3^{-3} R_2^{12[a,-b,c]}=1/H_{n}.
\] The desired result now follows by 
lemma \ref{invert111}. 
\end{proof}

\begin{corollary}
 The values $H(n)$ are real units.
\end{corollary}
\begin{proof}
 This is clear since $H(n)$ is real and the product of all roots of $p_n$ is $a_0=1$. 
\end{proof}

\begin{corollary}
 The polynomials $p_n(x)$ have the following simplified form:
\[
  p_n (x) =\prod_{[a,b,c] \in \mathrm{Cl}(\mathcal{O})} 
\left( x-  3^{-3} R_2(\tau_0)^{-12 [a,-b,c]} \right) 
\left( x-  3^{3} R_2(\tau_0)^{12 [a,-b,c]} \right),
\]
\end{corollary}

We have seen that $H_n$ is not a class invariant. But the quantity $A_n=H_n+\frac{1}{H_n}$ 
is a class invariant as we can verify using theorem \ref{shimura1}. This new invariant is 
not a unit anymore. 

The minimal polynomial $q_n\in \Z[x]$ of $A_n$ is given by 
\[
  q_n (x) =\prod_{[a,b,c] \in \mathrm{Cl}(\mathcal{O})} 
\left( x-  3^{-3} R_2(\tau_0)^{-12 [a,-b,c]} -3^{3} R_2(\tau_0)^{12 [a,-b,c]} \right).
\]
 
In Table \ref{qnpoly1} we give minimal polynomials $q_n$ 
for $19\leq n \leq 451$, $n\equiv 19 \bmod 24$. 

Observe that if $p_n=\sum_{\nu=0}^{2h} a_\nu x^\nu$ and 
$q_n=\sum_{\nu=0}^{h} b_\nu x^\nu$ then 
$b_\nu=a_{h+\nu}$ as one can prove using the Vieta formul\ae.

\begin{table}[H] \label{pnpoly1}
\caption{Polynomials $p_n$ for $19\leq n \leq 451$, $n\equiv 19 \bmod 24$.}
\[
{\tiny
\begin{array}{|l|c|l|}
\hline
n & \mbox{C.N.} & p_n(x)\\
\hline
19 & 1 &
x^2  -302x + 1 \\
\hline
43 & 1 & 
x^2  -33710x + 1 \\
\hline
67 & 1 & 
x^2  -1030190x + 1 \\
\hline
91 & 2 &
x^4 -17590492x^3 + 148475718 x^2 + -17590492x  +1 \\
\hline
115 &2 &
x^4 -210267100x^3 + 424646982x^2 -210267100x + 1 \\
\hline
139 & 3 & 
x^6 -1960891530x^5 -13943617329x^4 -30005622092x^3 -13943617329x^2  \\
& & -1960891530x +1\\
\hline
163 & 1 &
x^2 -15185259950x +1 \\
\hline
187 & 2 & 
x^4 -101627312860x^3 + 1102664076102x^2 -101627312860x + 1 \\
\hline
211 & 3 &
x^6 -604100444298x^5 + 20137792248015x^4 -414952590867788x^3 \\
& & + 20137792248015x^2 -604100444298x + 1\\
\hline 
235 & 2  &
x^4 -3253104234460x^3 + 47263043424582x^2 -3253104234460x + 1 \\
\hline 
259 & 4 &
x^8 -16106786824376x^7 -810131323637348x^6 -9925794993033992x^5 + \\
 & &26425093196592454x^4 -9925794993033992x^3 -810131323637348x^2  \\
 & & -16106786824376x + 1
\\
\hline 
283& 3 &
x^6 -74167114012170x^5 -119654555118897x^4 -3009681130315340x^3 \\
& & -119654555118897 x^2 -74167114012170x +1 
\\
\hline 
307 & 3 & 
x^6 -320508447128970x^5 -1963936794491697x^4 
-5740503875332940x^3 \\
& &-1963936794491697x^2 -320508447128970x + 1
\\
\hline
331 & 3 & 
x^6 -1309395837485706x^5 +113317118488006863x^4 
-11556648519941425484x^3  \\
& &+ 113317118488006863x^2 -1309395837485706x + 1
\\
\hline 
355 & 4 &
x^8 -5087640031882040x^7 + 583328538578918044x^6 -16479665770932342920x^5 \\
& & + 172809183517820572486x^4 -16479665770932342920x^3 \\
  && +583328538578918044x^2 -5087640031882040x + 1\\
\hline 
379 & 3 & 
x^6 -18895199824010634x^5 -4124999225954564913x^4 -274501369688142310220x^3 \\
&&-4124999225954564913x^2 -18895199824010634x +1
\\
\hline
403 & 2 & 
x^4 -67361590779141340x^3 + 
361802623368357702x^2 -67361590779141340x + 1\\
\hline
427 & 2  &
x^4 -231347688320676700x^3 + 2519902537964728902x^2 -231347688320676700
x + 1\\
\hline 
451 & 6 &
x^{12} -767819046799630740x^{11} + 161913605740919729922x^{10}  \\
&& -66458029641477066911812x^9 
-1654687781430584516238609x^8 \\
&&  -33875537641085268651117096x^7 \\
&&  + 81879258106346356247143452x^6 -33875537641085268651117096x^5 \\
&&  -1654687781430584516238609x^4 -66458029641477066911812x^3 \\
&&  + 161913605740919729922x^2 -767819046799630740x + 1\\
\hline 
\end{array}
}
\]
\end{table}

\begin{table}[H] \label{qnpoly1}
\caption{Polynomials $q_n$ for $19\leq n \leq 451$, $n\equiv 19 \bmod 24$.}
\[
{\tiny
\begin{array}{|l|c|l|}
\hline
n & \mbox{C.N.} & p_n(x)\\
\hline
19 & 1 &
x  -302 \\
\hline
43 & 1 & 
x  -33710  \\
\hline
67 & 1 & 
x  -1030190x  \\
\hline
91 & 2 &
x^2 -17590492x + 148475718  \\
\hline
115 &2 &
x^2 -210267100x + 424646982 \\
\hline
139 & 3 & 
x^3 -1960891530x^2 -13943617329x -30005622092\\
\hline
163 & 1 &
x -15185259950\\
\hline
187 & 2 & 
x^2 -101627312860x + 1102664076102 \\
\hline
211 & 3 &
x^3 -604100444298x^2+ 20137792248015x -414952590867788\\
\hline 
235 & 2  &
x^2 -3253104234460x + 47263043424582 \\
\hline 
259 & 4 &
x^4 -16106786824376x^3 -810131323637348x^2 -9925794993033992x + \\
 & &26425093196592454
\\
\hline 
283& 3 &
x^3 -74167114012170x^2 -119654555118897x -3009681130315340
\\
\hline 
307 & 3 & 
x^3 -320508447128970x^2 -1963936794491697x 
-5740503875332940
\\
\hline
331 & 3 & 
x^3 -1309395837485706x^2 +113317118488006863x 
-11556648519941425484
\\
\hline 
355 & 4 &
x^4 -5087640031882040x^3 + 583328538578918044x^2 -16479665770932342920x \\
& & + 172809183517820572486\\
\hline 
379 & 3 & 
x^3 -18895199824010634x^2 -4124999225954564913x -274501369688142310220x^3
\\
\hline
403 & 2 & 
x^2 -67361590779141340x + 
361802623368357702\\
\hline
427 & 2  &
x^2 -231347688320676700x + 2519902537964728902\\
\hline 
451 & 6 &
x^{6} -767819046799630740x^{5} + 161913605740919729922x^{4}  \\
&& -66458029641477066911812x^3 
-1654687781430584516238609x^2 \\
&&  -33875537641085268651117096x   + 81879258106346356247143452\\
\hline 
\end{array}
}
\]
\end{table}

\subsection{Some Unit Computations}
\label{3.1}
Professor P. Stevenhagen proposed to us the study of the following situation:
Consider the field $R_n:=\mathbb{Q}(H_n) \subset \mathbb{R}$. The field $R_n$ is 
an abelian extension of degree $2h$ of $\Q$ where $h$ is the class number of the 
order  $\Z[\theta_n]$. If $h=1$, i.e., when $n=19,43,67,163$ then $R_n$ is a 
real quadratic extension of $\Q$ and we can verify using magma that $H_n$ is 
a fundamental unit.

%
%
The field $R_n$ has $\mathbb{Q}(\sqrt{3n})$ as a subfield and we might ask 
if $\mathrm{N}_{R_n/\Q(\sqrt{3n})}H_n$ is a fundamental unit of the 
real quadratic field  $\mathbb{Q}(\sqrt{3n})$.
Using magma we compute that this is not always the case. In the following table we give the 
index of the subgroup  generated by $\mathrm{N}_{R_n/\Q(\sqrt{3n})}H_n$  
inside the group generated by the fundamental unit:
\begin{center}
{
\begin{tabular}{||c||l|l|l|l|l|l|l|l|l|l||}
\hline
\hline
N &19 & 43 & 67 & 91 & 115 & 139 & 163 & 187 & 211 &235  \\
\hline
index & 1 &1 & 1 & 2 & 2 &1 & 1 & 2 & 1 & 2 \\
\hline
\hline
N  & 259 & 283 & 307 & 331 &
355 & 379 & 403 & 427 & 451  & 475 \\
\hline
index &  4 & 1 & 1 & 3 & 2 & 1 & 2 & 2 & 2 & 8\\
\hline 
\hline
\end{tabular} 
}
\end{center}
The authors could not find an obvious pattern for the behavior of the index. 
Professor J. Antoniadis pointed to us the following pattern:  If the index is one then $N$ is prime but not 
vice versa since $331$ gives index $3$. We have checked that this is correct for all 
values of $N< 1000$.

\section{Computing class invariants for the ${3 \bmod 24}$ case}
\label{3mod24}

In this case the prime $2$ remains inert in $\mathcal{O}$ while $3$ ramifies. 
We compute first the structure of the group $\left(\mathcal{O}/72\mathcal{O}\right)^*$.
Modulo $72$ we have the following  values $n=3,27,51$ that are equivalent to $3$ modulo $24$.
The structure of the group $(\mathcal{O}/72 \mathcal{O})^*$ is computed to be the following:
\begin{center}
 \begin{tabular}{|c|c|c|}
\hline
  $n$ & $ (\mathcal{O}/9 \mathcal{O})^*$ & $(\mathcal{O}/8 \mathcal{O})^*$ \\
\hline
  3    & $\Z/6\Z \times \Z/3\Z \times \Z/3\Z$ & $ \Z/12\Z \times \Z/2\Z \times \Z/2\Z$ \\
  27  & $\Z/18\Z  \times \Z/3\Z $ & $\Z/12\Z \times \Z/2\Z \times \Z/2\Z$ \\
  51  & $\Z/18\Z  \times \Z/3\Z$  & $\Z/12\Z \times \Z/2\Z \times \Z/2\Z$ \\
\hline
 \end{tabular}
\end{center}
 The actions of the generators $\tau_i$  of each direct cyclic summand on the elements 
$\mathfrak{g}_0,\mathfrak{g}_1,\mathfrak{g}_2,\mathfrak{g}_3$ for each  case  is computed to be:\\
$\mathbf{n=3\bmod 72}$
\[
\begin{array}{|c||r|r|r|r|r|r|r|}
\hline  & \tau_{1} & \tau_{2} & \tau_{3} & \tau_{4} & \tau_{5} & \tau_{6}\\
\hline \tau(\mathfrak{g}_{0}) & {\mathfrak{g}_{0}} & -{\mathfrak{g}_{0}} & {\mathfrak{g}_{0}} & (-{\zeta}^{18}+{\zeta}^{6}){\mathfrak{g}_{1}} & -{\zeta}^{12}{\mathfrak{g}_{0}} & ({\zeta}^{12}-1){\mathfrak{g}_{0}}\\
\tau(\mathfrak{g}_{1}) & {}-{\mathfrak{g}_{1}} & -{\mathfrak{g}_{1}} & {\mathfrak{g}_{1}} & (-{\zeta}^{12}+1){\mathfrak{g}_{3}} & -{\zeta}^{12}{\mathfrak{g}_{1}} & -{\zeta}^{12}{\mathfrak{g}_{1}}\\
\tau(\mathfrak{g}_{2}) & {}-{\mathfrak{g}_{2}} & -{\mathfrak{g}_{2}} & {\mathfrak{g}_{2}} & -{\mathfrak{g}_{2}} & {\mathfrak{g}_{2}} & -{\zeta}^{12}{\mathfrak{g}_{2}}\\
\tau(\mathfrak{g}_{3}) & {}-{\mathfrak{g}_{3}} & -{\mathfrak{g}_{3}} & {\mathfrak{g}_{3}} & -{\zeta}^{18}{\mathfrak{g}_{0}} & -{\zeta}^{12}{\mathfrak{g}_{3}} & ({\zeta}^{12}-1){\mathfrak{g}_{3}}\\\hline \end{array}\]
$\mathbf{n=27 \bmod 72}$
\[
\begin{array}{|c||r|r|r|r|r|r|r|}
\hline  & \tau_{1} & \tau_{2} & \tau_{3} & \tau_{4} & \tau_{5} & \tau_{6}\\
\hline \tau(\mathfrak{g}_{0}) & {\mathfrak{g}_{0}} & -{\mathfrak{g}_{0}} & {\mathfrak{g}_{0}} & {\zeta}^{6}{\mathfrak{g}_{1}} & -{\zeta}^{12}{\mathfrak{g}_{0}} & ({\zeta}^{12}-1){\mathfrak{g}_{0}}\\
\tau(\mathfrak{g}_{1}) & {}-{\mathfrak{g}_{1}} & -{\mathfrak{g}_{1}} & {\mathfrak{g}_{1}} & {\zeta^{12}}{\mathfrak{g}_{3}} & -{\zeta}^{12}{\mathfrak{g}_{1}} & -{\zeta}^{12}{\mathfrak{g}_{1}}\\
\tau(\mathfrak{g}_{2}) & {}-{\mathfrak{g}_{2}} & -{\mathfrak{g}_{2}} & {\mathfrak{g}_{2}} & -{\mathfrak{g}_{2}} & {\mathfrak{g}_{2}} & -{\zeta}^{12}{\mathfrak{g}_{2}}\\
\tau(\mathfrak{g}_{3}) & {}-{\mathfrak{g}_{3}} & -{\mathfrak{g}_{3}} & {\mathfrak{g}_{3}} & -{\zeta}^{18}{\mathfrak{g}_{0}} & -{\zeta}^{12}{\mathfrak{g}_{3}} & ({\zeta}^{12}-1){\mathfrak{g}_{3}}\\\hline \end{array}\]
$\mathbf{n=51 \bmod 72}$
\[
\begin{array}{|c||r|r|r|r|r|r|r|}
\hline  & \tau_{1} & \tau_{2} & \tau_{3} & \tau_{4} & \tau_{5} & \tau_{6}\\
\hline \tau(\mathfrak{g}_{0}) & {\mathfrak{g}_{0}} & -{\mathfrak{g}_{0}} & {\mathfrak{g}_{0}} & {\zeta^{18}}{\mathfrak{g}_{1}} & -{\zeta}^{12}{\mathfrak{g}_{0}} & ({\zeta}^{12}-1){\mathfrak{g}_{0}}\\
\tau(\mathfrak{g}_{1}) & {}-{\mathfrak{g}_{1}} & -{\mathfrak{g}_{1}} & {\mathfrak{g}_{1}} & -{\mathfrak{g}_{3}} & -{\zeta}^{12}{\mathfrak{g}_{1}} & -{\zeta}^{12}{\mathfrak{g}_{1}}\\
\tau(\mathfrak{g}_{2}) & {}-{\mathfrak{g}_{2}} & -{\mathfrak{g}_{2}} & {\mathfrak{g}_{2}} & -{\mathfrak{g}_{2}} & {\mathfrak{g}_{2}} & -{\zeta}^{12}{\mathfrak{g}_{2}}\\
\tau(\mathfrak{g}_{3}) & {}-{\mathfrak{g}_{3}} & -{\mathfrak{g}_{3}} & {\mathfrak{g}_{3}} & {\zeta}^{18}{\mathfrak{g}_{0}} & -{\zeta}^{12}{\mathfrak{g}_{3}} & ({\zeta}^{12}-1){\mathfrak{g}_{3}}\\\hline \end{array}\]
Now if we raise the $\mathfrak{g}_i$ functions to $12$ we observe that $(\mathcal{O}/72 \mathcal{O})^*$
acts like a $3$-cycle on $\mathfrak{g}_i^{12}$ leaving $\mathfrak{g}_2^{6}$ invariant. 
Therefore $\mathfrak{g}_2^{12}$ gives rise to  a class invariant but also the functions 
$\mathfrak{g}_0^{12}+\mathfrak{g}_1^{12}+\mathfrak{g}_3^{12}$ give rise to class invariants but since 
their sum is $-36$ the $\mathfrak{g}_2^{12}$ is more interesting (it involves computation of 
only one modular function). 
In table  \ref{aa33} we present some small class polynomials 
for the invariant $\mathfrak{g}_2^{12}$.
Notice that $\mathfrak{g}_2^6$ is a also a class invariant but its minimal polynomial 
 does not have  coefficients in $\Z$. In table \ref{bb33} we  present some of the 
minimal polynomials of $\mathfrak{g}_2^6$ that are in $\Z[\sqrt{D'}][x]$, 
where $D'$ is the core discriminant of $D$, i.e., the non-square part of $D$.

\begin{table}[H] 
\caption{Polynomials for the invariant $\mathfrak{g}_2^{12}$, $n\equiv 3 \bmod 24$. \label{aa33}}
{\tiny
\begin{tabular}{|l|l|l|}
\hline 
n & C.N. & polynomials \\
\hline
3 & 1 & x+27 \\
\hline
27 & 1 & $x + 243$ \\
\hline
51 & 2 &
$
x^2
 + 1817 x
 + 63408
$ \\
\hline
75 & 2 & $ x^2
 + 8694 x
 + 729 $ \\
\hline
99 & 2 & 
$x^2
 + 33538 x
 + 675212
$
\\
\hline
133 & 2 & 
$
x^2
 + 110682 x
 + 3982527$\\
\hline
147 & 2 & $
x^2
 + 326646 x
 + 729$ 
\\
\hline
171 & 4 & 
$
x^4
 + 885577 x^3
 + 75449123 x^2
 + 1878791197 x
 + 9480040943
$ \\ \hline 
195  &4 & 
$
x^4
 + 2243057 x^3
 + 134435463 x^2
 + 2044439302 x
 + 4021471722$
\\ \hline
219 & 4 & 
$
x^4
 + 5374182 x^3
 + 177358410 x^2
 + 3337735739 x
 + 452759$
\\ \hline
243 & 3 & 
$
x^3
 + 12288753 x^2
 - 36669429 x
 + 129140163$
\\ \hline
267 & 2 & $
x^2
 + 27000090 x
 + 972001215$
\\ \hline
291 & 4 &
$
x^4
 + 57302460 x^3
 + 6191231603 x^2
 + 190393837000 x
 + 2422188$
\\ \hline 
315 & 4 & $
x^4
 + 117966740 x^3
 + 5465452595 x^2
 - 18078266775 x
 - 2283511958571$
\\ \hline 
339 & 6 & 
$
x^6
 + 236380194 x^5
 + 16297323547 x^4
 + 865456023300 x^3
 + 28951950717535 x^2$ \\
& &
$
 + 379087533199695 x
 + 3423896293014081$
\\ \hline 
363 & 4 & 
$
x^4
 + 462331692 x^3
 + 22863777174 x^2
 + 337039803468 x
 + 531441$
\\ \hline 
387 & 4 & 
$
x^4
 + 884736829 x^3
 + 65027878839 x^2
 + 1219285304855 x
 + 878209991853$
\\ \hline 
411 & 6 & 
$
x^6
 + 1659823938 x^5
 + 299376470893 x^4
 + 17714533511043 x^3
 + 122181573194844 x^2$ \\ & &
 $- 5409428705176675 x
 + 70928211329527433$
\\ \hline 
\end{tabular}
}
\end{table}

\begin{table}[H] 
\caption{Polynomials for the invariant $\mathfrak{g}_2^{6}$, $n\equiv 3 \bmod 24$. \label{bb33}}
{\tiny
\begin{tabular}{|r|c|l|}
\hline
D & C.N. & polynomials \\
\hline
-3 &
1&
$
x
 - 3 \sqrt{D'}
$ \\ \hline
-27 &
1&
$
x
 - 9 \sqrt{D'}
$ \\ \hline
-51 &
2&
$
x^2
 - 6 \sqrt{D'} x
 - 27
$ \\ \hline
-75 &
2&
$
x^2
 - 54 \sqrt{D'} x
 - 27
$ \\ \hline
-99 &
2&
$
x^2
 - 54 \sqrt{D'} x
 + 729
$ \\ \hline
-123 &
2&
$
x^2
 - 30 \sqrt{D'} x
 - 27
$ \\ \hline
-147 &
2&
$
x^2
 - 330 \sqrt{D'} x
 - 27
$ \\ \hline
-171 &
4&
$
x^4
 - 216 \sqrt{D'} x^3
 - 486 x^2
 - 19683
$ \\ \hline
-195 &
4&
$
x^4
 - 108 \sqrt{D'} x^3
 - 15714 x^2
 + 2916 \sqrt{D'} x
 + 729
$ \\ \hline
-219 &
4&
$
x^4
 - 156 \sqrt{D'} x^3
 + 22302 x^2
 + 4212 \sqrt{D'} x
 + 729
$ \\ \hline
-243 &
3&
$
x^3
 - 2025 \sqrt{D'} x^2
 - 6561 x
 + 6561 \sqrt{D'}
$ \\ \hline
-267 &
2&
$
x^2
 - 318 \sqrt{D'} x
 - 27
$ \\ \hline
-291 &
4&
$
x^4
 - 444 \sqrt{D'} x^3
 - 32130 x^2
 + 11988 \sqrt{D'} x
 + 729
$ \\ \hline
-315 &
4&
$
x^4
 - 1836 \sqrt{D'} x^3
 - 7290 x^2
 - 78732 \sqrt{D'} x
 + 531441
$ \\ \hline
-339 &
6&
$
x^6
 - 834 \sqrt{D'} x^5
 + 293355 x^4
 + 123444 \sqrt{D'} x^3
 - 7920585 x^2
 - 607986 \sqrt{D'} x
 - 19683
$ \\ \hline
-363 &
4&
$
x^4
 - 12420 \sqrt{D'} x^3
 - 218754 x^2
 + 335340 \sqrt{D'} x
 + 729
$ \\ \hline
-387 &
4&
$
x^4
 - 4536 \sqrt{D'} x^3
 - 486 x^2
 - 19683
$ \\ \hline

\end{tabular}
}
\end{table}

\section{An Application to Elliptic Curve Generation}

An important application of class invariants is that their minimal polynomials can be used for the
generation of elliptic curves over finite fields. In particular, a method called {\em Complex Multiplication}
 or {\em CM method} 
is used being raised from the theory of Complex Multiplication (CM) of elliptic curves
over the rationals. 
In the case of prime
fields $\mathbb{F}_p$, 
the CM method starts with the specification of
a discriminant value $D$, the determination of the order $p$ of
the underlying prime field and the order $m$ of the elliptic curve (EC). It then
computes the Hilbert polynomial,
which is uniquely determined by $D$ and locates one of its roots
modulo $p$. This root can be used to construct the parameters of
an EC with order $m$ over the field $\mathbb{F}_{p}$.

Alternative classes of polynomials can be used in the CM method as long as there is a transformation of their roots
modulo $p$ to the roots of the corresponding Hilbert polynomials. Clearly, polynomials $q_n$ can be used in the
CM method. However, firstly we have to find a relation between the $j$-invariant and the values $A_n$. 
Using \cite[lemma\ 3]{KoKo2} we obtain the following relation between the $j$-invariant $j_n$ and the Ramanujan values $t_n$:
\begin{equation} \label{relate1}
j_n = (t_n^6-27t_n^{-6}-6)^3. 
\end{equation}
If we set $C=t_n^6-27t_n^{-6}$ then we easily observe that 
\begin{equation}
\label{c-eq}
C^2 = 27(A_n - 2). 
\end{equation}
\begin{remark}
 An other way to obtain a formula that relates the values $H_n$, $t_n$ is working with equation 
(\ref{gequ}) that relates $\mathfrak{g}_i^{12}$ to $j$. We compute the coefficients of the polynomial
$\prod_{0\leq i < j \leq 3} (X-\mathfrak{g}_i^{12}\mathfrak{g}_j^{12})$. These are symmetric polynomials in 
the variables $\mathfrak{g}_i^{12}$ and can be expressed as polynomials on the elementary 
symmetric polynomials
which  up to sign are the coefficients of the polynomial in eq. (\ref{gequ}). Using this approach with 
 magma we arrive at
 \begin{eqnarray*}
G(Y,j) &= &
   {Y}^{6}-270\,{Y}^{5}+ \left( -36\,j+26487 \right) {Y}^{4}+ \\
 & &  \left( -{j}
^{2}+1512\,j-1122660 \right) {Y}^{3} + \left( -26244\,j+19309023
 \right) {Y}^{2}\\
 &  &-143489070\,Y+387420489.                  
 \end{eqnarray*}
 We arrive at the same polynomial if we eliminate $t_n^6$ from eq. (\ref{relate1}) and 
the definition of $H_n$. 
\end{remark}

If we wish to construct an EC over a prime field $\mathbb{F}_p$ using the $q_n$ polynomials, we have to
 find one of their roots modulo $p$ and then transform it to a root of the corresponding 
Hilbert polynomial $j_n$. The root $j_n$ can be acquired using Eq. ~(\ref{c-eq}) and the relation $j_n = (C-6)^3$.

Let us give a brief example on how $q_n$ polynomials can be used in the CM method. 
Suppose that we wish to generate an EC over the prime field $\mathbb{F}_p$ where 
$p=2912592100297027922366637171900365067697538262949$ and we 
choose to use a discriminant equal to $n=259$. Initially, 
the CM method having as input the values $p$ and $n$ constructs the order of the
 EC which is equal to a prime number \[m=2912592100297027922366635123877214056291799441739.\]
 Then, the polynomial $q_{259}(x)$ is constructed
\begin{eqnarray*}
q_{259}(x)  & = &  x^4-16106786824376x^3-810131323637352x^2  \\ & &-9877474632560864x+ 28045355843867152. 
\end{eqnarray*}
This polynomial has four roots modulo $p$. Every such root can be transformed to a root $j_{n}$ 
using Eq. ~(\ref{c-eq}) and the relation $j_n = (C-6)^3$. Eq. ~(\ref{c-eq})
 will result to two values $C$ and this means that for every root modulo $p$ of the $q_{259}(x)$
 polynomial we will have two roots $j_{259}$. However, only one of these two roots 
gives rise to an EC with order $m$. 
The correct choice is made easily: we follow the steps of the CM method, construct an EC having as input a value 
$j_{259}$ and then check whether the resulted curve (or its twist) has indeed order $m$. 
If the answer is negative, then this value is rejected.

For example, one root modulo $p$ of the $q_{259}(x)$ polynomial is equal to 
\[
r = 292000143869356471233943284623526736899256758497. 
\]
Using Eq. ~(\ref{c-eq}), we compute the two solutions 
\[
C_1 = 1555795526891231123931549739786994193545044932499
\] 
and 
\[
C_2 = 1356796573405796798435087432113370874152493330450
\]
and therefore the two possible values of the $j$-invariant are 
\[
j_1 = 2662539171725102375366109856465433412332472450493
\]
and 
\[
j_2 = 1859938916666171899538097507602720023646246323886. 
\]
Selecting the first value $j_1$, we construct the EC $y^2 = x^3 + ax + b$ where 
\[
a = 1545339657951389136173847270246016180230953846699
\]
and 
\[
b = 59362405201916783327019122863889097588123143483.
\]
In order to check if this EC (or its twist) has order $m$, we choose a point $P$ at random in
it and we compute the point 
$Q=mP$. If this point is equal to the point at infinity then the EC has order $m$. Making the necessary computations for the above EC, we see that this is the case and our construction is finished. Thus, we conclude that we have chosen the correct $j$-invariant and the second value $j_2$ is rejected.

\providecommand{\bysame}{\leavevmode\hbox to3em{\hrulefill}\thinspace}
\providecommand{\MR}{\relax\ifhmode\unskip\space\fi MR }
\providecommand{\MRhref}[2]{%
  \href{http://www.ams.org/mathscinet-getitem?mr=#1}{#2}
}
\providecommand{\href}[2]{#2}

\end{document}